\RequirePackage{fix-cm}
\documentclass[english]{article}
\usepackage[T1]{fontenc}
\usepackage[latin9]{inputenc}
\usepackage{geometry}
\usepackage{prettyref}
\usepackage{amsthm}
\usepackage{amsmath}
\usepackage{amssymb}
\usepackage{graphicx}
\usepackage{setspace}
\usepackage[all]{xy}
\makeatletter
\numberwithin{equation}{section}
\numberwithin{figure}{section}
\theoremstyle{plain}
\newtheorem{thm}{\protect\theoremname}[section]
  \theoremstyle{remark}
  \newtheorem{rem}[thm]{\protect\remarkname}
  \theoremstyle{plain}
  \newtheorem{cor}[thm]{\protect\corollaryname}
  \theoremstyle{plain}
  \newtheorem{prop}[thm]{\protect\propositionname}
  \theoremstyle{plain}
  \newtheorem{lem}[thm]{\protect\lemmaname}

\usepackage[perpage,para,symbol]{footmisc}
\usepackage[all]{xy}
\usepackage{cancel}
\usepackage[leftcaption]{sidecap}
\usepackage{graphicx}

\sidecaptionvpos{figure}{c}

\newcommand{\FigBesBeg}[1][1.0]{%
 \let\MyFigure\figure
 \let\MyEndfigure\endfigure
 \renewenvironment{figure}[1]{\begin{SCfigure}[#1]##1}{\end{SCfigure}}}

\newcommand{\FigBesEnd}{%
 \let\figure\MyFigure
 \let\endfigure\MyEndfigure}

\DefineFNsymbols*{lamportnostar}[math]{\dagger\ddagger\S\P\|{\dagger\dagger}{\ddagger\ddagger}}
\setfnsymbol{lamportnostar}

\newcommand{\F}{{\mathbf{F}}}

\theoremstyle{plain}
\newtheorem{ques}[thm]{Question}
\theoremstyle{plain}
\newtheorem{claim}[thm]{\protect\claimname}

\usepackage{enumitem}
\setlist{leftmargin=*}

\makeatother

\usepackage{babel}
  \providecommand{\corollaryname}{Corollary}
  \providecommand{\lemmaname}{Lemma}
  \providecommand{\propositionname}{Proposition}
  \providecommand{\remarkname}{Remark}
\providecommand{\theoremname}{Theorem}

\begin{document}

\title{Growth of Primitive Elements in Free Groups}

\author{D. Puder%
\thanks{Supported by the Adams Fellowship Program of the Israel Academy of
Sciences and Humanities, and by the ERC.%
}~~ and~~ C. Wu}
\maketitle
\begin{abstract}
In the free group $F_{k}$, an element is said to be primitive if
it belongs to a free generating set. In this paper, we describe what
a generic primitive element looks like. We prove that up to conjugation,
a random primitive word of length $N$ contains one of the letters
exactly once asymptotically almost surely (as $N\to\infty$).

This also solves a question from the list `Open problems in combinatorial
group theory' {[}Baumslag-Myasnikov-Shpilrain 02'{]}. Let $p_{k,N}$
be the number of primitive words of length $N$ in $F_{k}$. We show
that for $k\geq3$, the exponential growth rate of $p_{k,N}$ is $2k-3$.
Our proof also works for giving the exact growth rate of the larger
class of elements belonging to a proper free factor.
\end{abstract}
2010 Mathematics Subject Classifi{}cation: 20E05 (Primary) 05A16 (Secondary)

\tableofcontents{}

\section{Introduction\label{sec:Introduction}}

Let $F_{k}$ be the free group on $k$ generators $X=\left\{ x_{1},\ldots,x_{k}\right\} $
($k\geq2$). Elements in $F_{k}$ are represented by reduced words
in the alphabet $X^{\pm1}=\left\{ x_{1,}^{\pm1}x_{2,}^{\pm1}\cdots,x_{k}^{\pm1}\right\} $.
A word $w\in F_{k}$ is called \emph{primitive} if it belongs to some
free generating set\textbf{. }We let $P_{k,N}$\marginpar{$P_{k,N}$}
denote the set of primitive elements of word length $N$ in $F_{k}$.
It is known (see, for example, \cite{BMS02}) that as $N\to\infty$
the set of primitive words is exponentially small in $F_{k}$. Namely,
the exponential growth rate\marginpar{$\rho_{k}$} 
\[
\rho_{k}\stackrel{\mathrm{{\scriptstyle def}}}{=}\limsup_{N\to\infty}\sqrt[N]{\left|P_{k,N}\right|}
\]
is strictly smaller than that of the whole free group $F_{k}$, which
is $2k-1$. As observed in \cite{rivin2004remark}, $\rho_{2}=\sqrt{3}$,
which gives the only case where the growth rate is known. For $k\geq3$,
various upper bounds on $\rho_{k}$ have been established \cite{BV02counting,BMS02,shpilrain2005counting}.
The best upper bound to date is due to Shpilrain \cite{shpilrain2005counting}
who showed $\rho_{k}\leq\lambda_{k}$, where $\lambda_{k}$ is the
greatest real root of $\lambda\left(\lambda^{2}-1\right)\left(\lambda-\left(2k-2\right)\right)+1$.
Here $\lambda_{k}<2k-2$ for each $k$, but $\lambda_{k}$ approaches
$2k-2$ in the limit. A simple lower bound of $\rho_{k}\geq2k-3$
stems from the fact that every word of the form $x_{1}w(x_{2},x_{3},\cdots,x_{k})$,
where $w$ is a word of length $N-1$ in $\left\{ x_{2}^{\pm1},\ldots,x_{k}^{\pm1}\right\} $,
forms a free generating set together with $\left\{ x_{2},x_{3},\cdots,x_{k}\right\} $,
hence is primitive. 

The exact value of $\rho_{k}$ is the content of one of the open questions,
attributed to M. Wicks, in \cite[Problem F17]{baumslag2002open} (see
also the active website \cite[Problem F19]{OpenProblems}). Here we
answer the question and show the following tight result:
\begin{thm}
\label{thm:growth-of-prims}For all $k\geq3$, 
\[
\rho_{k}=\lim_{N\to\infty}\sqrt[N]{\left|P_{k,N}\right|}=2k-3.
\]
Moreover, there are positive constants $c_{k}$ and $C_{k}$ such
that 
\[
c_{k}\cdot N\cdot\left(2k-3\right)^{N}\leq\left|P_{k,N}\right|\leq C_{k}\cdot N\cdot\left(2k-3\right)^{N}.
\]
\end{thm}
\begin{rem}
The second statement of Theorem \ref{thm:growth-of-prims} can be
sharpened to $\left|P_{k,N}\right|=\left(1+o_{N}\left(1\right)\right)\cdot\widehat{C_{k}}\cdot N\cdot\left(2k-3\right)^{N}$
for a specific constant $\widehat{C_{k}}$ which can be computed.
This can be inferred from Theorem \ref{thm:growth-of-cyclic-prims}
and the analysis in Proposition \ref{prop:prim-to-conj} below.
\end{rem}
The above theorem follows from an analysis of conjugacy classes of
primitives in free groups. A word $w=a_{1}a_{2}\cdots a_{N}$ is called
\emph{cyclically reduced} if $a_{1}\neq a_{N}^{-1}$. Such words,
up to a cyclic permutation of letters, uniquely represent conjugacy
classes in $F_{k}$. Hence for $w\in F_{k}$ we call the conjugacy
class $[w]$\marginpar{$[w]$} \emph{the cyclic word }associated with\emph{
$w$}. Let the \emph{cyclic length} of $w$, denoted by $\left|w\right|_{c}$\marginpar{$\left|w\right|_{c}$},
be the length of the cyclically reduced representatives of $[w]$.

There is a stark difference between the behavior of $P_{2,N}$ and
that of $P_{k,N}$ when $k\geq3$: whereas in $F_{2}$ `most' long
primitives are conjugates of short ones, it turns out that for higher
rank free groups the generic primitive word is nearly cyclically-reduced.
In particular, the growth of the set of primitive elements is the
same as that of primitive conjugacy classes (cyclic words) with respect
to cyclic length. (This is the content of Proposition \ref{prop:prim-to-conj}
below.)

Consider the set \marginpar{$C_{k,N}$}
\[
C_{k,N}=\left\{ \left[w\right]\,\middle|\, w\in F_{k}\,\,\mathrm{is\,\, primitive\,}\,\mathrm{and}\,\,\left|w\right|_{c}=N\right\} .
\]
We compare the size of $C_{k,N}$ with its subset of cyclic-words
in which some letter $x\in X$ appears exactly once (either as itself
or its inverse), namely the set \marginpar{$L_{k,N}$}
\[
L_{k,N}=\left\{ \left[w\right]\,\middle|\,\mathrm{some}\, x\in X\,\mathrm{appears\,\, in}\,\, w\,\,\mathrm{exactly\,\, once}\right\} \subseteq C_{k,N}.
\]
The size of $L_{k,N}$ can be easily approximated as%
\footnote{This expression is very close to the truth, except that we double
count words in which two or more letters appear exactly once. The
exact cardinality of $L_{k,N}$ can be obtained by an application
of the inclusion-exclusion formula. Note that the share of doubly-counted
words is exponentially negligible in $L_{k,N}$: it is of exponential
order $\left(2k-5\right)^{N}$.%
} $\left|L_{k,N}\right|\approx2k\left(2k-2\right)\left(2k-3\right)^{N-2}.$
So that 
\[
\limsup_{N\to\infty}\sqrt[N]{\left|L_{k,N}\right|}=2k-3.
\]

\begin{thm}
\label{thm:growth-of-cyclic-prims}For $k\geq2$
\[
\limsup_{N\to\infty}\sqrt[N]{\left|C_{k,N}\right|}=2k-3.
\]
For $k\geq3$,
\[
\limsup_{N\to\infty}\sqrt[N]{\left|C_{k,N}\setminus L_{k,N}\right|}<2k-3.
\]
Moreover, 
\[
\left|C_{k,N}\right|=\left(1+o_{N}\left(1\right)\right)\cdot\frac{2k\left(2k-2\right)}{\left(2k-3\right)^{2}}\left(2k-3\right)^{N}.
\]
 
\end{thm}
The second statement of theorem \ref{thm:growth-of-cyclic-prims}
means that except for an exponentially small set, all primitive cyclic-words
contain one of the letters exactly once. When $k\geq3$ the first
and last statements are an immediate consequence of the second one
and the approximated size of $L_{k,N}$ as given above.

Note that the first statement of theorem \ref{thm:growth-of-cyclic-prims}
is also valid for $k=2$: the exponential growth rate of conjugacy
classes of primitives in $F_{2}$ is 1. This special case was already
shown in \cite[Prop 1.4]{myasnikov2003automorphic}: it turns out
the size of $C_{2,N}$ is exactly $4\varphi\left(n\right)$, where
$\varphi\left(\cdot\right)$ is the Euler function. Whereas $\rho_{2}=\sqrt{3}$
is strictly larger than $1$, for all $k\geq3$ the growth of primitive
cyclic-words is the same as the growth of primitive words. 

A natural question along the same vein would be to estimate the growth
of the larger set $S_{k,N}$\marginpar{$S_{k,N}$} consisting of words
in $F_{k}$ which are contained in a proper free factor (clearly,
$P_{k,N}\subseteq S_{k,N}$). Our proof of Theorem \ref{thm:growth-of-prims}
also applies to this question and yields that $S_{k,N}$ has the same
exponential growth rate as $P_{k,N}$: 
\begin{cor}
\label{cor:growth-of-proper-free-factors}For $k\geq3$ we have
\[
\lim_{N\to\infty}\sqrt[N]{\left|S_{k,N}\right|}=\lim_{N\to\infty}\sqrt[N]{\left|S_{k,N}\setminus P_{k,N}\right|}=2k-3.
\]

\end{cor}
We show that $\lim_{N\to\infty}\sqrt[N]{\left|S_{k,N}\right|}\leq2k-3$
in Section \ref{sub:The-growth-of-non-primitives-in-free-factors}.
This requires only a small variation on the proof of Theorem \ref{thm:growth-of-prims}.
The lower bound is, again, easier, and follows immediately from the
fact that primitives are exponentially negligible in $F_{k}$ (this
fact follows from Theorem \ref{thm:growth-of-prims} but also, as
mentioned above, from previous results concerning the growth of primitives).
Indeed, this fact shows that most words in any size $k-1$ subset
of the letters are non-primitive. We conclude that the number of non-primitive
words in $S_{k,N}$ grows at least as fast as $\left(2k-3\right)^{N}$\label{page:lower-bound-for-Skn}.
Thus $S_{k,N}$ is indeed larger than $P_{k,N}$ in a non-negligible
manner, namely, 
\[
\lim_{N\to\infty}\sqrt[N]{\left|S_{k,N}\setminus P_{k,N}\right|}\geq2k-3=\lim_{N\to\infty}\sqrt[N]{\left|P_{k,N}\right|}.
\]

\begin{rem}
For a different proof showing that $\lim_{N\to\infty}\sqrt[N]{\left|S_{k,N}\setminus P_{k,N}\right|}=2k-3$,
see \cite[Thm 8.2]{Pud14b} (in the terminology therein, every word
in $S_{k,N}\setminus P_{k,N}$ has primitivity rank $\leq k-1$).
In the techniques of that paper (especially \cite[Prop. 4.3]{Pud14b}),
it can be shown that a generic word in $S_{k,N}\setminus P_{k,N}$
is, up to conjugation, a word in some $\left(k-1\right)$-subset of
the letters of $X$.\\

\end{rem}
Our proofs rely on a thorough analysis of the Whitehead algorithm
to detect primitive elements. To a lesser extent, we also use a characterization
of primitive elements based on the distribution they induce on finite
groups. In Section \ref{sec:Whitehead} we give some background on
Whitehead algorithm and describe the graphs used in it, called Whitehead
graphs. We then divide the set of primitives into finitely many classes
according to certain properties of their Whitehead graphs. Most of
these classes turn out to be of negligible size, but we postpone the
somewhat technical proof of this fact to Section \ref{sec:Growth-of-partitions}.
In Section \ref{sec:exp-growth-rate} we give some background on the
aforementioned ``statistical'' characterization of primitives, estimate
the size of the remaining classes of primitives and complete the proofs
of Theorems \ref{thm:growth-of-prims}, \ref{thm:growth-of-cyclic-prims}
and of Corollary \ref{cor:growth-of-proper-free-factors}. We end
with some open questions in Section \ref{sec:Open-Questions}.

\section*{Acknowledgements}

We would like to thank Tsachik Gelander for bringing our attention
to the question. We thank Warren Dicks, Ilya Kapovich, Nati Linial,
Shahar Mozes and Alexey Talambutsa for beneficial comments. We also
thank the anonymous referee for his valuable comments. The second
author would like to thank the Hebrew University for providing hospitality
and stimulating mathematical environment during which part of this
work was conducted.

\section{Whitehead Graphs\label{sec:Whitehead}}

\noindent In \cite{Whi36a}, Whitehead introduced the first algorithm
to detect primitive words in $\F_{k}$ (and more generally subsets
of bases of $\F_{k}$). (Subsequently, in \cite{Whi36b}, he solved
a more general question: Given two words $w_{1},w_{2}\in F_{k}$,
when does there exist an automorphism $\phi\in\mathrm{Aut}\left(F_{k}\right)$
mapping $w_{1}$ to $w_{2}$? Note that $w_{1}$ is primitive if and
only if there is an automorphism mapping it to a single-letter word.)
Along the years it has become the most standard way of detecting primitive
elements. Stallings generalized the algorithm in order to detect words
belonging to free factors of $\F_{k}$ \cite{Sta99}. For other algorithms
to detect primitives see, e.g., \cite[Chapter I.2]{LS70} or \cite{Pud14a}.

The algorithm is based on the following construction: Let $M_{k}$
be a 3-manifold which is the connected sum of $k$ copies of $\mathbb{S}^{1}\times\mathbb{S}^{2}$.
Clearly, we have $\pi_{1}(M_{k})=F_{k}$. Fix a set of $k$ disjoint
2-spheres $S_{1},S_{2},\cdots,S_{k}$, one corresponding to each summand,
so that $\widehat{M}_{k}=M_{k}\backslash\bigcup_{i=1}^{k}S_{i}$ is
simply connected with $2k$ boundary components $S_{1}^{+},S_{1}^{-},S_{2}^{+},\cdots,S_{k}^{-}$.
The manifold $M_{k}$ may be visualized as the double of a handlebody
$H_{k}$ with $\left\{ S_{i}\right\} $ being the double of a cut
system of $H_{k}$ (a cut system is a set of disjoint discs that cuts
the handlebody into a ball). For every $w\in F_{k}=\pi_{1}(M_{k})$,
the cyclic word $\left[w\right]$ can be realized as a simple curve
in $M_{k}$. Conversely, given any oriented curve in $M_{k}$ one
can write down a cyclic word in $F_{k}$ by reading off the sequence
of spheres the curve intersects, with signs. Hence we get a bijective
correspondence between cyclic words $[w]$ and homotopy classes of
oriented simple curves in $M_{k}$. 

\begin{figure}[h]
\includegraphics[scale=0.72]{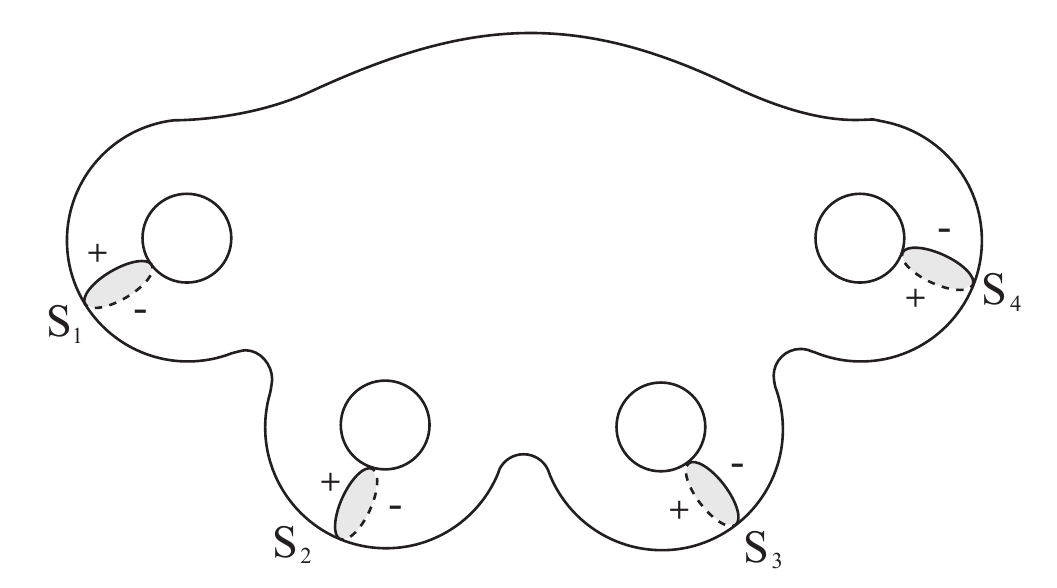}\includegraphics[scale=0.72]{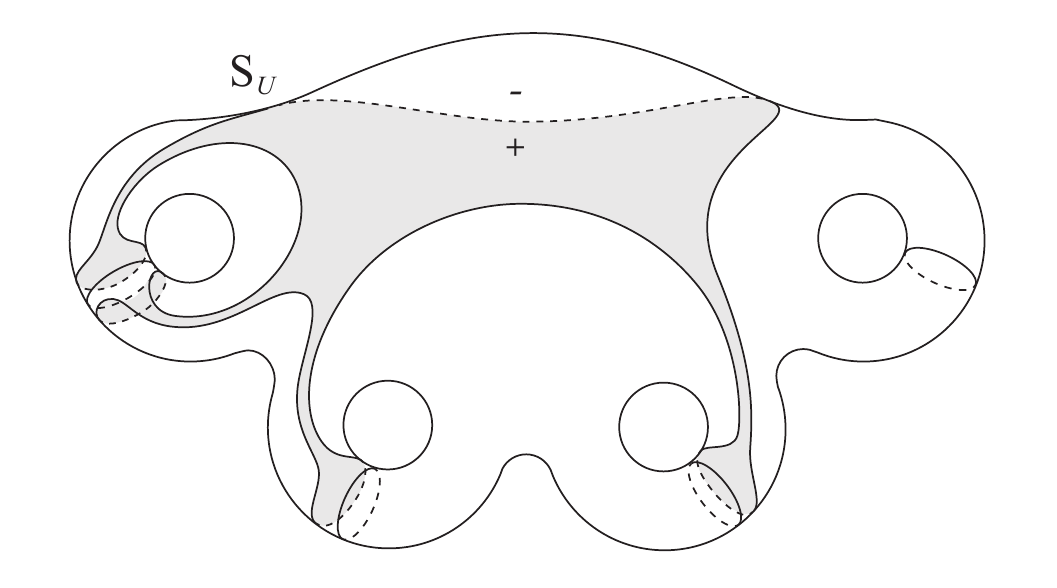}\caption{Spheres $S_{i}$ in $H_{k}$ and $S_{\mathcal{{U}}}$ with $\mathcal{{U}}=\left\{ S_{1}^{+},S_{1}^{-},S_{2}^{+},S_{3}^{-}\right\} $}
\end{figure}

Given any proper non-empty subset $\mathcal{U}\subset\left\{ S_{1}^{+},S_{1}^{-},\cdots,S_{k}^{-}\right\} $,
there is an embedded 2-sphere $S_{\mathcal{U}}$ in $\widehat{M}_{k}$
separating the boundary components in $\mathcal{{U}}$ from those
not in $\mathcal{{U}}$. For every $v\in X^{\pm1}$ denote by $S_{v}$
the corresponding boundary component of $\widehat{M}$ (so $S_{x_{i}}=S_{i}^{+}$
and $S_{x_{i}^{-1}}=S_{i}^{-}$). If there exists some $v=x_{j}^{\,\varepsilon}\in X^{\pm1}$
such that $S_{v}\notin\mathcal{{U}}$ and $S_{v^{-1}}\in\mathcal{{U}}$
then $S_{\mathcal{{U}}}$ is an essential non-separating sphere%
\footnote{Namely, a non-contractible embedding of a sphere which does not separate
$\widehat{M}_{k}\cup S_{j}$ into two connected components.%
} in $\widehat{M}_{k}\cup S_{j}$. The \textbf{Whitehead automorphism}
$\varphi_{(\mathcal{{U}},v)}$ of $F_{k}$ is then defined by replacing
the sphere $S_{j}$ by $S_{\mathcal{{U}}}$ and writing each cyclic
word as the intersection pattern of the corresponding curve with the
new set of spheres. In the example illustrated above, $S_{3}^{+}\notin\mathcal{{U}}$
and $S_{3}^{-}\in\mathcal{{U}}$ hence we may replace $S_{3}$ with
$S_{\mathcal{{U}}}$. Writing down $\varphi_{(\mathcal{{U}},v)}$
formally one gets:
\begin{itemize}
\item $\varphi_{(\mathcal{{U}},v)}(v)=v$ ; $\varphi_{(\mathcal{{U}},v)}(v^{-1})=v^{-1}$;
\item for $u\neq v,$

\begin{itemize}
\item $\varphi_{(\mathcal{{U}},v)}(u)=u$ if $S_{u},S_{u^{-1}}\notin\mathcal{{U}}$;
\item $\varphi_{(\mathcal{{U}},v)}(u)=vuv^{-1}$ if $S_{u},S_{u^{-1}}\in\mathcal{{U}}$;
\item $\varphi_{(\mathcal{{U}},v)}(u)=vu$ and $\varphi_{(\mathcal{{U}},v)}(u^{-1})=uv^{-1}$
if $S_{u}\in\mathcal{{U}}$, $S_{u^{-1}}\notin\mathcal{{U}}$;
\end{itemize}
\end{itemize}
By forgetting the order in which the spheres are intersected and looking
only at the arcs connecting boundary components in $\widehat{M}_{k}$
one gets a finite graph with $2k$ vertices labeled $X^{\pm1}=\left\{ x_{1}^{\pm1},\cdots,x_{k}^{\pm1}\right\} $.
This is called the \textbf{Whitehead graph} of the cyclic word $[w]$,
denoted by\marginpar{$\Gamma\left(w\right)$} $\Gamma(w)$. For example,
$\Gamma\left(w\right)$ for $w=x_{1}x_{2}^{\,2}x_{3}^{-1}x_{2}^{-2}\in F_{3}$
is:

\[
\xymatrix{x_{1}\ar@{-}[dr] & x_{1}^{-1}\\
x_{2}\ar@{-}[d]\ar@{-}@/^{0.4pc}/[r] & x_{2}^{-1}\ar@{-}@/^{0.4pc}/[l]\ar@{-}[u]\\
x_{3} & x_{3}^{-1}\ar@{-}[ul]
}
\]
Going from manifolds to graphs, to every $\left(\mathcal{{U}},v\right)$
defined as above corresponds a partition of the vertices by $Z=\left\{ x_{u}\,|\, S_{u}\in\mathcal{{U}}\right\} $
and $Y=X^{\pm1}\backslash(Z\cup\left\{ v\right\} )$. Denote $\phi_{Y,Z,v}=\varphi_{(\mathcal{{U}},v)}$\marginpar{$\phi_{Y,Z,v}$},
and notice that $X^{\pm1}=Y\amalg Z\amalg\left\{ v\right\} $. The
following theorem, part of the foundation for Whitehead's algorithm,
plays a central role in our argument:
\begin{thm}
\label{thm:in-free-factor-then-cut-vertex}\cite[Thm 2.4]{Sta99}
If $w$ is contained in a proper free factor of $F_{k}$, then $\Gamma\left(w\right)$
has a cut vertex.
\end{thm}
Namely, there exists a vertex $v$ such that $\Gamma(w)\backslash\left\{ v\right\} $
is disconnected. This includes the case where $\Gamma\left(w\right)$
is itself disconnected. Note that, in particular, all primitive elements
are contained in a rank one free factor, hence have Whitehead graphs
with cut vertices.

Note that the cyclic length $\left|w\right|_{c}$ of $w$ is the number
of edges in the Whitehead graph corresponding to a cyclically reduced
representative. A natural candidate for a length reducing Whitehead
automorphism is therefore to replace the sphere corresponding to the
cut vertex $v$ by one that separates a connected component of $\Gamma(w)\backslash\left\{ v\right\} $.
Indeed, we have the following:
\begin{prop}
\label{prop:white-uto-reduces}\cite[Prop 2.3]{Sta99} Let $v$ be
a cut-vertex of $\Gamma\left(w\right)$, and let $Y$ and $Z$ be
a non-trivial partition of the remaining vertices so that there are
no edges between $Y$ and $Z$, and $v^{-1}\in Z$. Then
\[
\left|\phi_{Y,Z,v}\left(w\right)\right|_{c}=\left|w\right|_{c}-E\left(Y,v\right).
\]

\end{prop}
Here $E\left(Y,v\right)$ is the number of edges connecting $v$ to
$Y$. For instance, for $w=x_{1}x_{2}^{\,2}x_{3}^{-1}x_{2}^{-2}$
as above, there are two possible cut-vertices: $x_{2}$ and $x_{2}^{-1}$.
If one chooses $v=x_{2}$ and $Z=\left\{ x_{1},x_{1}^{-1},x_{2}^{-1}\right\} $,
then $[\phi_{Y,Z,v}\left(w\right)]=[x_{2}x_{1}x_{2}x_{3}^{-1}x_{2}^{-2}]=[x_{1}x_{2}x_{3}^{-1}x_{2}^{-1}]$
has cyclic length $4$.

Moreover, it is easy to see that if $w$ is contained in a proper
free factor then it is almost always possible to find a triplet $\left(Y,Z,v\right)$
as in Proposition \ref{prop:white-uto-reduces} with $E\left(Y,v\right)>0$:
the only exceptions are $\left|w\right|_{c}\leq1$ or when $w$ is
a word in a proper subset of the letters, say $x_{1},\ldots,x_{j}$
($j<k$), and it does not belong to a proper free factor in $F\left(x_{1},\ldots,x_{j}\right)$.
This is the crux of the Whitehead algorithm to detect primitives:
since the second case cannot occur for primitive elements with $\left|w\right|_{c}>1$,
if $w$ is primitive one can always apply a sequence of Whitehead
automorphisms according to cut vertices in the Whitehead graph, until
it becomes a (conjugate of a) single-letter word. 

Our proof of Theorem \ref{thm:growth-of-prims} (and of Corollary
\ref{cor:growth-of-proper-free-factors}) relies on a rigorous analysis
of the possible triplets $\left(Y,Z,v\right)$. We say that a triplet
$\left(Y,Z,v\right)$ is \emph{valid} for the cyclic word $\left[w\right]$
if it satisfies the statement in Proposition \ref{prop:white-uto-reduces}
(namely, if $v$ is a cut-vertex of $\Gamma\left(w\right)$, $Y$
and $Z$ are a non-trivial partition of the remaining vertices with
$E\left(Y,Z\right)=0$, and $v^{-1}\in Z$). Let $A_{Y,Z,v}$ denote
the set of all cyclic words having $(Y,Z,v)$ as a valid triplet;
namely\marginpar{$A_{Y,Z,v}$}
\[
A_{Y,Z,v}=\left\{ \left[w\right]\,|\,\left(Y,Z,v\right)\textrm{ is a valid triplet\,\ for }w\right\} 
\]
and\marginpar{$A_{Y,Z,v}^{N}$}

\[
A_{Y,Z,v}^{N}=\left\{ \left[w\right]\in A_{Y,Z,v}\,\middle|\,\left|w\right|_{c}=N\right\} .
\]
By Theorem \ref{thm:in-free-factor-then-cut-vertex},
\[
C_{k,N}\subseteq\bigcup_{\left(Y,Z,v\right)}A_{Y,Z,v}^{N}
\]
taking the union over all possible triplets partitioning $X^{\pm1}$
(with $Y,Z\ne\emptyset$ and $v^{-1}\in Z$). 

We proceed by bounding the growth of primitives in $A_{Y,Z,v}$ for
each of the finitely many triplets $(Y,Z,v)$. Intuitively, the cut
vertex and partition will restrict the number of possible ways to
connect vertices, hence result in a smaller growth rate. In the extreme
case, if both sets $Y$ and $Z$ contain roughly half of the elements
of $X^{\pm1}$, namely $|Y|\approx k$ and $|Z|\approx k$; then from
any vertex in $Y$ one can connect only to another vertex in $Y\cup\left\{ v\right\} $,
resulting in $\sim k$ choices. If we ignore the possibility of going
through $v$, then the possible number of such cyclic words would
be roughly only $k^{N}$, which amounts to an exponential growth rate
of $k$ (note that this applies to the whole set $A_{Y,Z,v}$ and
not only to the primitives in it). Hence we should expect $A_{Y,Z,v}$
to grow faster for triplets $(Y,Z,v)$ where one of $Y,Z$ is almost
all of $X^{\pm1}$. 

Indeed it turns out that $A_{Y,Z,v}$ is negligible unless one of
$Y,Z$ is very small:
\begin{prop}
\label{prop:Every-triplet-except...-is-negligible}Every triplet $\left(Y,Z,v\right)$
satisfies
\[
\limsup_{N\to\infty}\sqrt[N]{\left|A_{Y,Z,v}^{N}\right|}<2k-3,
\]
unless $min(|Y|,|Z|)=1$ or $Y=\left\{ x,x^{-1}\right\} $ for some
letter $x$.
\end{prop}
The proof of this proposition involves some careful analysis in various
cases, and we postpone it to Section \ref{sec:Growth-of-partitions}.
In Section \ref{sec:exp-growth-rate} we assume this proposition,
give the precise growth rates for the remaining essential partitions
and obtain our theorems.

\section{Proof of Theorems\label{sec:exp-growth-rate}}

\noindent In this section we complete the proofs of Theorems \ref{thm:growth-of-prims}
and \ref{thm:growth-of-cyclic-prims} and of Corollary \ref{cor:growth-of-proper-free-factors}.

\subsection{Primitives and cyclic primitives\label{sub:The-growth-of-primitives}}

Here, we present the observation that, unlike in $F_{2}$ , for $k\geq3$
it suffices to count conjugacy classes containing primitive words. 
\begin{prop}
\label{prop:prim-to-conj} For $k\geq3$, if $\left|C_{k,N}\right|\leq C\cdot(2k-3)^{N}$
for some $C>0$ as follows from Theorem \ref{thm:growth-of-cyclic-prims},
then 
\[
|P_{k,N}|\leq D\cdot N\cdot(2k-3)^{N}
\]
for some $D>0$.\end{prop}
\begin{proof}
Each $w\in P_{k,N}$ is of the form 
\[
uw'u^{-1}
\]
where $w'\in F_{k}$ is cyclically reduced and primitive. Let $\ell$
be the word length of $u$, so that $0\leq\ell\leq\frac{N-1}{2}$
and $\left|w\right|_{c}=\left|w'\right|_{c}=N-2\ell$.

Since $w'$ is primitive, in particular, it is not a proper power,
hence each of its cyclic shifts is different. Namely, the cyclic word
$[w']$ is represented by exactly $N-2\ell$ distinct cyclically reduced
words. On the other hand, $u$ can be any word of length $\ell$ as
long as the first letter of $u^{-1}$ and the last letter of $u$
do not cancel out their adjacent letters in $w'$. There are $\left(2k-1\right)^{\ell-1}\left(2k-2\right)$
such words. Therefore, 
\begin{eqnarray*}
\left|P_{k,N}\right| & = & \sum_{\ell=0}^{\left\lfloor \frac{N-1}{2}\right\rfloor }\left(N-2\ell\right)\left|C_{k,N-2\ell}\right|\left(2k-2\right)\left(2k-1\right)^{\ell-1}\\
 & \leq & N\cdot\sum_{\ell=0}^{\left\lfloor \frac{N-1}{2}\right\rfloor }\left|C_{k,N-2\ell}\right|\left(2k-1\right)^{\ell}\\
 & \leq & N\cdot\sum_{\ell=0}^{\left\lfloor \frac{N-1}{2}\right\rfloor }C\cdot\left(2k-3\right)^{N-2\ell}\left(2k-1\right)^{\ell}\\
 & = & N\cdot C\cdot\left(2k-3\right)^{N}\sum_{\ell=0}^{\left\lfloor \frac{N-1}{2}\right\rfloor }\left(\frac{2k-1}{\left(2k-3\right)^{2}}\right)^{\ell}.
\end{eqnarray*}
For $k\geq3$, $\left(\frac{2k-1}{\left(2k-3\right)^{2}}\right)\leq\frac{{5}}{9}<1$.
Bounding the geometric series, we deduce that $\left|P_{k,N}\right|\leq D\cdot N\cdot\left(2k-3\right)^{N}$.
\end{proof}
The proposition shows that Theorem \ref{thm:growth-of-prims} follows
from Theorem \ref{thm:growth-of-cyclic-prims}: For the lower bound
in Theorem \ref{thm:growth-of-prims} recall that the number of cyclically
reduced primitive words of length $N$ with one of the letters appearing
exactly once is 
\[
N\cdot\left|L_{k,N}\right|=\left(1-o_{N}\left(1\right)\right)\cdot N\cdot2k\cdot\left(2k-2\right)\cdot\left(2k-3\right)^{N-2}.
\]
To complete the proofs of Theorems \ref{thm:growth-of-prims} and
\ref{thm:growth-of-cyclic-prims}, it remains to bound from above
the growth of cyclic primitives. Before starting the proof we present
in Section \ref{sub:Ingredients-for-the-growth-of-cyclic-primitives}
a couple of useful facts which will be used in the sequel.

\subsection{Ingredients for bounding cyclic primitives\label{sub:Ingredients-for-the-growth-of-cyclic-primitives}}

First, we give some background on a line of thought regarding primitive
words which is different from Whitehead's and leads to a measure-theoretic
characterization of primitives. Let $w=x_{i_{1}}^{\,\varepsilon_{1}}x_{i_{2}}^{\,\varepsilon_{2}}\cdots x_{i_{N}}^{\,\varepsilon_{N}}$
be a word in $F_{k}$. For every group $G$, $w$ induces a \emph{word
map }from the Cartesian product $G^{k}$ to $G$, by substitutions:
\[
w:\left(g_{1},\ldots,g_{k}\right)\mapsto g_{i_{1}}^{\,\varepsilon_{1}}g_{i_{2}}^{\,\varepsilon_{2}}\cdots g_{i_{N}}^{\,\varepsilon_{N}}.
\]
When $G$ is finite (compact) and $G^{k}$ is given the uniform (Haar,
resp.) measure, the push forward by $w$ of this measure results in
a new measure on $G$, which we denote by $G_{w}$\marginpar{$G_{w}$}.
It is an easy observation that if $w_{1}$ and $w_{2}$ are in the
same $\mathrm{Aut}\, F_{k}$-orbit of $F_{k}$, then they induce the
same measure on every finite or compact group, namely $G_{w_{1}}=G_{w_{2}}$
(see \cite[Observation 1.2]{PP15}). In particular, if $w$ is primitive,
then $G_{w}=G_{x_{1}}$ which is clearly the uniform (Haar) measure
on $G$. 

It is natural to ask whether the converse also holds. Namely, if $G_{w_{1}}=G_{w_{2}}$
for every finite (compact) group, does it imply that $w_{1}$ and
$w_{2}$ are in the same $\mathrm{Aut}\, F_{k}$-orbit? This conjecture
is still wide open. However, the special case concerning primitives
was settled in \cite{PP15}. It is shown there that if $G_{w}$ is
uniform for every finite group $G$, then $w$ is primitive. In the
heart of the argument in \cite{PP15} lies a result about the distributions
induced by words on the symmetric groups $S_{n}$. We re-formulate
it as follows:
\begin{thm}
\cite[Thm 1.7]{PP15} Let $w\in F_{k}$. For every $n\in\mathbb{N}$
let $\sigma_{w,n}$ be a random permutation in $S_{n}$ distributed
according to $\left(S_{n}\right)_{w}$. Then $w$ is non-primitive
if and only if there exists some $n_{0}$ such that for all $n>n_{0}$
we have

\[
\mathbb{{E}}(|\mathrm{Fix}\,(\sigma_{w,n})|)>1,
\]
where $\mathrm{Fix}\mbox{\,(\ensuremath{\sigma})}$ denotes the set
of fixed points of $\sigma$.
\end{thm}
Note that for $w$ primitive, $\sigma_{w,n}$ is a uniformly distributed
random permutation in $S_{n}$, hence the expected number of fixed
points is exactly $1$. From this theorem we derive the following
fact which will be useful in the argument. We say that $w_{1}$ and
$w_{2}$ are \emph{letter disjoint} words if their reduced forms use
disjoint subsets of the alphabet $X$.
\begin{prop}
\label{prop: w1w2-prim-then-one-of-them-is-prim} Let $w_{1},w_{2}\in F_{k}$
be letter disjoint words. If the concatenation $w_{1}w_{2}$ is primitive,
then at least one of $w_{1}$ or $w_{2}$ is primitive.\end{prop}
\begin{proof}
Since the push-forward measure by $w$ is a class function, the probability
$Pr(\sigma_{w,n}(i)=i)$ is independent of $i$ (here $i\in\left\{ 1,\ldots,n\right\} $),
and likewise, $Pr(\sigma_{w,n}(i)=j)$ is independent of $i$ and
$j$ as long as $i\ne j$. Thus, $\mathbb{{E}}(|\mathrm{Fix}(\sigma_{w,n})|)>1$
if and only if $Pr(\sigma_{w,n}(1)=1)>\frac{1}{n}$. Let $p(w,n)=Pr(\sigma_{w,n}(1)=1)$. 

Since $w_{1}$ and $w_{2}$ are letter disjoint, they induce independent
push-forward measures on $S_{n}$. If both words are non-primitive
then for large enough $n$, both $p(w_{1},n)>\frac{1}{n}$ and $p(w_{2},n)>\frac{1}{n}$,
which implies
\begin{eqnarray*}
p\left(w_{1}w_{2},n\right) & = & Pr(\sigma_{w_{1}w_{2},n}(1)=1)\\
 & = & \sum_{j=1}^{n}Pr(\sigma_{w_{1},n}(1)=j)\cdot Pr(\sigma_{w_{2},n}(j)=1)\\
 & = & p(w_{1},n)p(w_{2},n)+\left(n-1\right)\frac{1-p(w_{1},n)}{n-1}\cdot\frac{1-p(w_{2},n)}{n-1}\\
 & = & \frac{1}{n}+\frac{n}{n-1}\cdot\left(p(w_{1},n)-\frac{1}{n}\right)\cdot\left(p(w_{2},n)-\frac{1}{n}\right)>\frac{1}{n}.
\end{eqnarray*}
This contradicts the assumption that $w_{1}w_{2}$ is primitive.
\end{proof}
Recall that by Proposition \prettyref{prop:Every-triplet-except...-is-negligible}
primitives from $A_{Y,Z,v}$ for most triplets $\left(Y,Z,v\right)$
are negligible. We make some simple observations about the remaining
three types of triplets:
\begin{itemize}
\item \noindent If $|Y|=1$, say $Y=\left\{ a\right\} $, and $w\in A_{Y,Z,v}$
then each appearance of $a$ is followed by $v^{-1}$ and each appearance
of $a^{-1}$ is preceded by $v$. It is not hard to see that the growth
rate here is at least $\sqrt{\left(2k-3\right)^{2}+1}>2k-3$. Indeed,
consider the $\left(2k-2\right)\left(2k-3\right)$ ordered reduced
pairs of letters not containing $a^{\pm1}$ and, in addition, the
pair $av^{-1}$. Each one of these pairs can be followed by one of
$\left(2k-3\right)\left(2k-3\right)+1$ of these pairs, which shows
the lower bound. Since every possible pair of letters is followed
by one of less than $\left(2k-1\right)^{2}$ possible pairs, the growth
rate is strictly less than $2k-1$. In fact, the exact growth rate
is the largest (real) root of $\lambda^{5}-\left(2k-3\right)\lambda^{4}-3\lambda^{3}+\left(2k-3\right)\lambda^{2}+3\lambda+\left(2k-3\right)$,
which tends to $2k-3$ as $k\to\infty$. \\

\item If $Y=\left\{ a,a^{-1}\right\} $ then every instance of $a^{\pm1}$
in a word from $A_{Y,Z,v}$ is in the form $\ldots va^{m}v^{-1}\ldots$
for some $0\ne m\in\mathbb{Z}$. The exponential growth rate here
is the largest (real) root of $\lambda^{4}-\left(2k-2\right)\lambda^{3}+\left(2k-4\right)\lambda^{2}+\left(2k-2\right)\lambda-6k+11$,
which again approaches $2k-3$ from above as $k\to\infty$. Again,
looking at pairs of letters one can easily infer the growth rate is
strictly less than $2k-1$.\\

\item Finally, if $|Z|=1$, namely $Z=\left\{ v^{-1}\right\} $, then $v^{-1}$
is followed only by $v^{-1}$, and $v$ is preceded only by $v$.
So $A_{Y,Z,v}$ consists of all cyclic words not containing $v^{\pm1}$
(together with $\left\{ v^{m}\,\middle|\, m\in\mathbb{Z}\right\} $).
Hence, the growth rate is exactly $\left(2k-3\right)$.
\end{itemize}
In particular, this analysis gives rise to the following naive bound:
\begin{cor}
\label{cor:Every-triplet-is-less-than-2k-1}Let $k\geq2$. Every triplet
$\left(Y,Z,v\right)$ satisfies
\[
\limsup_{N\to\infty}\sqrt[N]{\left|A_{Y,Z,v}^{N}\right|}<2k-1.
\]

\end{cor}

\subsection{Proof of Theorem \ref{thm:growth-of-cyclic-prims}\label{sub:Proof-of-Theorem}}

The moral of the following proof will be that even when the set $A_{Y,Z,v}$
has exponential growth rate larger than $2k-3$, every cyclic \emph{primitive}
$\left[w\right]\in A_{Y,Z,v}$ can be shortened `fast enough' by the
corresponding Whitehead automorphism $\phi_{Y,Z,v}$. Recall that
the second statement of Theorem \ref{thm:growth-of-cyclic-prims}
is that for $k\geq3$,
\begin{equation}
\limsup_{N\to\infty}\sqrt[N]{\left|C_{k,N}\setminus L_{k,N}\right|}<2k-3.\label{eq:c-kn-minus-L-kN}
\end{equation}
As mentioned in Section \ref{sec:Introduction}, when $k\geq3$ the
other two statements of the theorem follow from \prettyref{eq:c-kn-minus-L-kN},
and for $k=2$, the relevant statement ($\limsup_{N\to\infty}\sqrt[N]{\left|C_{2,N}\right|}=1$)
is already known \cite[Prop. 1.4]{myasnikov2003automorphic}. Therefore,
we shall prove \prettyref{eq:c-kn-minus-L-kN} by induction on $k$,
assuming only that $k\geq3$ and that for $k-1$ we have

\[
\limsup_{N\to\infty}\sqrt[N]{\left|C_{k-1,N}\right|}=2k-5.
\]

Assume then that $k\geq3$, and let $M_{k,N}\subseteq C_{k,N}\backslash L_{k,N}$\marginpar{$M_{k,N}$}
be the set of cyclic primitive words such that either
\begin{itemize}
\item $\left[w\right]\in A_{Y,Z,v}^{N}$ with $Y=\left\{ x\right\} $ and
$x$ appearing at least 4 times in $\left[w\right]$, or
\item $\left[w\right]\in A_{Y,Z,v}^{N}$ with $Y=\left\{ x,x^{-1}\right\} $
and $\left[w\right]$ containing at least $2$ instances of $vx^{m}v^{-1}$
(for any $0\ne m\in\mathbb{Z}$).
\end{itemize}
Let $M_{k,N}^{c}$\marginpar{$M_{k,N}^{c}$} denote the complement
of $M_{k,N}$ inside $C_{k,N}\backslash L_{k,N}$. We proceed by showing
that the exponential growth rates of $\left|M_{k,N}\right|$ and $|M_{k,N}^{c}|$
are both strictly less than $2k-3$.
\begin{lem}
\label{lem:growth-of-MckN}
\[
\limsup_{N\to\infty}\sqrt[N]{\left|M_{k,N}^{c}\right|}<2k-3.
\]
\end{lem}
\begin{proof}
In light of Proposition \ref{prop:Every-triplet-except...-is-negligible}
we only need to consider cyclic primitive words $[w]\in M_{k,N}^{c}$
with a Whitehead partition $(Y,Z,v)$ where $\min\left\{ |Y|,|Z|\right\} =1$
or $Y=\{x,x^{-1}\}$. 

If $|Z|=1$ then by definition $Z=\{v^{-1}\}$, and every cyclic word
in $A_{Y,Z,v}$ is either a power of $v$ or a word in the alphabet
$X^{\pm1}\setminus\left\{ v^{\pm1}\right\} $. The set $\{v^{n}\}$
is clearly negligible. In the latter case, $[w]$ is primitive in
the letters $X\setminus\left\{ v^{\pm1}\right\} $. It follows from
the induction hypothesis that the exponential growth rate of this
set of cyclic primitives is $2k-5<2k-3$ .

Assume next that $Y=\{x,x^{-1}\}$ and $[w]$ has exactly one instance
of $vx^{m}v^{-1}$ for some $\left|m\right|\geq2$ (the case $\left|m\right|=1$
is impossible as $M_{k,N}^{c}\cap L_{k,N}=\emptyset$). Pick a representative
of $\left[w\right]$ of the form $\left[x^{m}w'\right]$, where $w'$
has length $<N$ and does not contain the letter $x$. By Proposition
\ref{prop: w1w2-prim-then-one-of-them-is-prim}, $w'$ is primitive
in $F_{k-1}$ hence this set has exponential growth rate $2k-5$.

The remaining case is $Y=\{x\}$ and $x^{\pm1}$ appearing exactly
twice or thrice in $[w]$ (again, if $x^{\pm1}$ appears only once,
then $\left[w\right]$ belongs to $L_{k,N}$). Consider first the
case where $x^{\pm1}$ appears exactly \textbf{twice}:\\
We can write $\left[w\right]$ in the form
\[
\left[\left(u_{1}xu_{2}\right)^{\pm1}w_{1}\left(u_{1}xu_{2}\right)^{\pm1}w_{2}\right],
\]
where $u_{1},u_{2}$ are maximal sequence of letters preceding and
following both instances of $x^{\pm1}$, respectively. 

Let $\ell_{i}$ be the length of $u_{i}$. Up to a factor of $N^{3}$,
which is negligible in terms of exponential growth rates, we know
$\ell_{1}$, $\ell_{2}$, $\left|w_{1}\right|$ and $\left|w_{2}\right|$.
There are about $\left(2k-3\right)^{\ell_{1}+\ell_{2}}$ options for
the values of $u_{1}$ and $u_{2}$. The automorphism $\psi\in\mathrm{Aut}\left(\F_{k}\right)$
which maps $x\mapsto u_{1}^{-1}xu_{2}^{-1}$ and leaves unchanged
the remaining letters, maps $\left[w\right]$ to the primitive cyclic
word
\[
\left[w'\right]=\left[x^{\pm1}w_{1}x^{\pm1}w_{2}\right].
\]
Let $N'=N-2\left(\ell_{1}+\ell_{2}\right)$ be the length of $w'$.
We claim that the number of possible $w'$ is bounded above by $C\cdot\left(2k-3-\varepsilon\right)^{N'}$for
some $C,\varepsilon>0$, $\varepsilon$ small. This will suffice as
the number of possible $\left[w\right]$ is then bounded by some polynomial
in $N$ times
\[
\left(2k-3-\varepsilon\right)^{N'}\cdot\left(2k-3\right)^{\ell_{1}+\ell_{2}}=\left(2k-3-\varepsilon\right)^{N-2\ell_{1}-2\ell_{2}}\cdot\sqrt{2k-3}^{2\ell_{1}+2\ell_{2}}\leq\left(2k-3-\varepsilon\right)^{N}.
\]
Firstly, if one of $w_{1}$ or $w_{2}$ is trivial, then as in the
preceding case, the number of options for $\left[w'\right]$ is at
most some constant times $\left(2k-5\right)^{N'}$. So, assume $w_{1},w_{2}\ne1$.
The word $\left[w'\right]$ belongs to some $A_{Y',Z',v'}^{N'}$.
By the maximality of $u_{1}$ and $u_{2}$, each of the vertices $x$
and $x^{-1}$ in the Whitehead graph $\Gamma\left(w'\right)$ has
at least two neighbors in $X^{\pm1}\setminus\left\{ x,x^{-1}\right\} $.
Hence, $Y'\ne\left\{ x\right\} ,\left\{ x^{-1}\right\} ,\left\{ x,x^{-1}\right\} $.
Also, it is not possible that $v'=x^{\pm1}$ and $Z'=\left\{ \left(v'\right)^{-1}\right\} $,
since cyclic words corresponding to this triplet are either words
in $X\setminus\left\{ v'\right\} $ or powers of $v'$. Hence, the
triplet $\left(Y',Z',v'\right)$ induces some non-trivial partition
of $\left(X\setminus\left\{ x\right\} \right)^{\pm1}$ (both $Y'$
and $Z'$ intersect $ $$\left(X\setminus\left\{ x\right\} \right)^{\pm1}$).
Hence, $w_{1}$ and $w_{2}$ are words (albeit not cyclic) corresponding
to some non-trivial triplet partitioning $\left(X\setminus\left\{ x\right\} \right)^{\pm1}$.
But by Corollary \ref{cor:Every-triplet-is-less-than-2k-1}, the exponential
growth of such subsets is strictly less than $2\left(k-1\right)-1=2k-3$.

Finally, consider the case where $Y=\left\{ x\right\} $ and $x^{\pm1}$
appears exactly \textbf{three times} in $\left[w\right]$: \\
The proof that this subset grows slower than $\left(2k-3\right)^{N}$
is very similar to the previous case. This time, each such word is
of the form
\[
\left[\left(u_{1}xu_{2}\right)^{\pm1}w_{1}\left(u_{1}xu_{2}\right)^{\pm1}w_{2}\left(u_{1}xu_{2}\right)^{\pm1}w_{3}\right]
\]
with $u_{1},u_{2}$ maximal. It can be shortened via an automorphism
to 
\[
\left[w'\right]=\left[x^{\pm1}w_{1}x^{\pm1}w_{2}x^{\pm1}w_{3}\right]
\]
of length $N'$. Again, up to a polynomial factor of $N^{4}$ we know
$\ell_{1},\ell_{2},\left|w_{1}\right|,\left|w_{2}\right|$ and $\left|w_{3}\right|$,
and we claim that the number of options for $\left[w'\right]$ is
bounded by some constant times $\left(2k-3-\varepsilon\right)^{N'}$
with $\varepsilon>0$. Hence the total number of options for $\left[w\right]$
is bounded by some polynomial in $N$ times
\[
\left(2k-3-\varepsilon\right)^{N-3\ell_{1}-3\ell_{2}}\cdot\left(2k-3\right)^{\ell_{1}+\ell_{2}}\leq\left(2k-3-\varepsilon\right)^{N}
\]
for $\varepsilon$ small enough. Indeed, if two of $w_{1}$, $w_{2}$
and $w_{3}$ are trivial, we are again in the same situation as in
the case $Y=\left\{ x,x^{-1}\right\} $. Otherwise, the exact same
argument as before shows that $\left[w'\right]\in A_{Y',Z',v'}^{N'}$
for some triplet $\left(Y',Z',v'\right)$ partitioning $\left(X\setminus\left\{ x\right\} \right)^{\pm1}$
non-trivially. 

This covered all the cases of $M_{k,N}^{c}$ and hence the lemma is
established.
\end{proof}
Now we move on to the remaining set $M_{k,N}$. The idea is to shorten
such words by applying appropriate Whitehead automorphisms until the
result falls outside of $M_{k,*}$ (i.e.~outside of $M_{k,n}$ for
all $n$). To achieve this goal we first consider cyclic words $[w]\in M_{k,N}$
such that the corresponding automorphism $\phi_{Y,Z,v}$ maps them
into $L_{k,*}$: $[\phi_{Y,Z,v}(w)]\in L_{k,*}$ . Denote this subset
by $\widetilde{L}_{k,N}\subseteq M_{k,N}$\marginpar{$\widetilde{L}_{k,N}$}.
We claim that:
\begin{lem}
\label{clm:growth-of-LkN}
\[
\limsup_{N\to\infty}\sqrt[N]{\left|\widetilde{L}_{k,N}\right|}<2k-3.
\]
\end{lem}
\begin{proof}
Consider first the words $[w]\in\widetilde{L}_{k,N}$ with triplet
$(Y,Z,v)$ where $Y=\left\{ x\right\} $. The effect of $\phi=\phi_{Y,Z,v}$
on $\left[w\right]$ is precisely that: each instance of the form
$\ldots xv^{-1}\ldots$ becomes simply $\ldots x\ldots$, and each
instance of $\ldots vx^{-1}\ldots$ turns into $\ldots x^{-1}\ldots$.
In particular, the only letter whose number of appearances in $\left[w\right]$
is changed by $\phi$ is $v$. By definition, $\left[w\right]\in\widetilde{L}_{k,N}\subseteq M_{k,N}\subseteq C_{k,N}\setminus L_{k,N}$,
so the letter which $[\phi\left(w\right)]$ contains exactly once
is necessarily $v^{\pm1}$. Hence, aside for one, all occurrences
of $v^{\pm1}$ in $\left[w\right]$ are as part of either $xv^{-1}$
or $vx^{-1}$. We deduce that $\left[w\right]$ is of the form
\[
\left[vw'\right]
\]
with $w'$ being a word of length $N-1$ in $2\left(k-2\right)$ building
blocks of length $1$: $\left(X\setminus\left\{ x^{\pm1},v^{\pm1}\right\} \right)^{\pm1}$
and $2$ building blocks of length 2: $\left(xv^{-1}\right)^{\pm1}$.
(In other words, $w'$ is any word in $F_{k-1}$ but where one of
the letters is of length $2$). This kind of words clearly has exponential
growth rate strictly less than $2k-3$. (To be precise, the growth
rate is the larger root of $\lambda^{2}-\left(2k-4\right)\lambda-2$.)

The complement of this latter subset inside $\widetilde{L}_{k,N}$
consists of primitive cyclic words belonging to $A_{Y,Z,v}^{N}$ with
$Y=\left\{ x,x^{-1}\right\} $. This time, $\phi$ turns each instance
of $\ldots vx^{m}v^{-1}\ldots$ into $\ldots x^{m}\ldots$. The same
arguments as before show that if $\left[w\right]$ is such a word,
then
\[
\left[w\right]=\left[vw'\right]
\]
where $w'$ is composed of building blocks from $\left(X\setminus\left\{ x^{\pm1},v^{\pm1}\right\} \right)^{\pm1}$
together with $\left\{ vx^{m}v^{-1}\,\middle|\,0\ne m\in\mathbb{Z}\right\} $.
Every letter in $w'$ is followed by one of at most $\left(2k-4\right)$
possible letters, showing this type of words also has exponential
growth rate $<2k-3$, and thus completing the proof.\end{proof}
\begin{lem}
\label{lem:growth-of-MkN}
\[
\limsup_{N\to\infty}\sqrt[N]{\left|M_{k,N}\right|}<2k-3.
\]
\end{lem}
\begin{proof}
Every $[w]\in M_{k,N}$ is equipped with some triplet $(Y,Z,v)$ from
the definition of $M_{k,N}$ (so $Y=\left\{ x\right\} $ or $Y=\left\{ x,x^{-1}\right\} $
for some $x\in X^{\pm1}$). First, we observe that the corresponding
Whitehead automorphism $\phi_{Y,Z,v}$ shortens $\left[w\right]$
by at least 4. There are in total $2k\left(2k-2\right)$ triplets
with $Y=\left\{ x\right\} $ and $k\left(2k-2\right)$ triplets with
$Y=\left\{ x,x^{-1}\right\} $. Let $\mathcal{W}$ denote the set
of these $2k\left(2k-2\right)+k\left(2k-2\right)=6k\left(k-1\right)$
possible Whitehead automorphisms. 

In other words, for every $[w]\in M_{k,N}$ there exists $\phi_{1}\in\mathcal{{W}}$
such that $|\phi_{1}\left(w\right)|_{c}=N'\leq N-4$. If $[\phi_{1}\left(w\right)]\in M_{k,N'}$
we apply the corresponding automorphism $\phi_{2}\in\mathcal{{W}}$
and obtain a cyclic word of length $\leq N-8$. Since for all $n$
we have $C_{k,n}=L_{k,n}\cup M_{k,n}\cup M_{k,n}^{c}$, one can continue
this process until the resulting cyclic word is either in $M_{k,n}^{c}$
or in $\widetilde{L}_{k,n}$ for some $2\leq n\leq N$ (note that
each cyclic word in $M_{k,*}$ is of length $\geq8$). Let $\left[\widehat{w}\right]\in M_{k,n}^{c}\cup\widetilde{L}_{k,n}$
be the cyclic-word we obtain this way, i.e. 
\[
\left[w\right]\overset{\phi_{1}}{\to}\left[w_{1}\right]\overset{\phi_{2}}{\to}\left[w_{2}\right]\to\ldots\overset{\phi_{r}}{\to}\left[\widehat{w}\right].
\]
Each element in $M_{k,N}$ is uniquely determined by $r$, $n=\left|\widehat{w}\right|_{c}$,
$\left[\widehat{w}\right]\in M_{k,n}^{c}\cup\widetilde{L}_{k,n}$
and $(\phi_{1},\phi_{2},\cdots,\phi_{r})\in\mathcal{{W}}^{r}$, where
$r\leq R=\left\lfloor \frac{N-n}{4}\right\rfloor $. For each $n$,
the number of possible tuples $(\phi_{1},\phi_{2},\cdots,\phi_{r})$
is 
\[
\sum_{i=0}^{R}\left|\mathcal{{W}}\right|^{i}=\sum_{i=0}^{R}\left[6k\left(k-1\right)\right]^{i}\leq\frac{\left[6k\left(k-1\right)\right]^{R+1}-1}{6k\left(k-1\right)-1}\leq2\left[6k\left(k-1\right)\right]^{R}.
\]
By Lemmas \ref{lem:growth-of-MckN} and \ref{clm:growth-of-LkN},
the possible number of $\left[\widehat{w}\right]$ of length $n$
is bounded from above by $C\cdot\left(2k-3-\varepsilon\right)^{n}$
for some $C,\varepsilon>0$. Summing over all possible values of $n$
we obtain:
\begin{eqnarray*}
\left|M_{k,N}\right| & \leq & 2C\cdot\sum_{n=2}^{N}\left(2k-3-\varepsilon\right)^{n}\cdot\left[6k\left(k-1\right)\right]^{R}\\
 & \leq & 2C\cdot\sum_{n=2}^{N}\left(2k-3-\varepsilon\right)^{n}\cdot(\sqrt[4]{6k(k-1)})^{N-n}.
\end{eqnarray*}
For $k\geq3$ we can pick $\varepsilon$ small enough so that 
\[
\sqrt[4]{6k\left(k-1\right)}<2k-3-\varepsilon,
\]
hence 
\[
|M_{k,N}|<2C\cdot N\cdot\left(2k-3-\varepsilon\right)^{N}.
\]

\end{proof}
This completes the proof of \eqref{eq:c-kn-minus-L-kN}, hence also
of Theorem \ref{thm:growth-of-cyclic-prims}. Theorem \ref{thm:growth-of-prims}
now follows by Proposition \ref{prop:prim-to-conj}.

\subsection{The growth of non-primitives belonging to free factors \label{sub:The-growth-of-non-primitives-in-free-factors}}

Finally, let us say some words about the variation of the proof required
for Corollary \ref{cor:growth-of-proper-free-factors}. Recall that
$S_{k,N}$ denotes the set of words of length $N$ in $\F_{k}$ belonging
to a proper free factor. We ought to show that $\left|S_{k,N}\right|$
grows exponentially with base $\left(2k-3\right)$. We already mentioned
on Page \pageref{page:lower-bound-for-Skn} why $\left(2k-3\right)$
is a lower bound. To show it is also an upper bound, we repeat similar
arguments as above%
\footnote{In fact, the proof of Corollary \ref{cor:growth-of-proper-free-factors}
alone could be shorter than the proof of Theorem \ref{thm:growth-of-cyclic-prims}.
The same shorter proof would show that $\limsup_{N\to\infty}\sqrt[N]{C_{k,N}}=2k-3$.
In other words, much of the complexity of the analysis in Sections
\ref{sub:Ingredients-for-the-growth-of-cyclic-primitives} and \ref{sub:Proof-of-Theorem}
is required only for showing the stronger result that the growth rate
of $\left|C_{k,N}\setminus L_{k,N}\right|$ is strictly smaller than
$2k-3$.%
}:

Firstly, the same argument as in Section \ref{sub:The-growth-of-primitives}
shows the exponential growth of $S_{k,N}$ is the same as the exponential
growth of $\overline{S_{k,N}}$\marginpar{$\overline{S_{k,N}}$},
the set of cyclic-words of length $N$ belonging to a proper free
factor. By Theorem \ref{thm:in-free-factor-then-cut-vertex}, each
$\left[w\right]\in\overline{S_{k,N}}$ belongs to some $A_{Y,Z,v}^{N}$.
For most triplets, Proposition \ref{prop:Every-triplet-except...-is-negligible}
shows they grow slower than $\left(2k-3\right)^{N}$. When $Z=\left\{ v^{-1}\right\} $,
$A_{Y,Z,v}^{N}\subseteq\overline{S_{k,N}}$ and grows exponentially
with base $\left(2k-3\right)$.

Consider next words in $A_{Y,Z,v}^{N}$ such that either
\begin{itemize}
\item $Y=\left\{ x,x^{-1}\right\} $ and there is exactly one instance of
$vx^{m}v^{-1}$, or 
\item $Y=\left\{ x\right\} $ and there are up to three instances of $x^{\pm1}$. 
\end{itemize}
It is evident that this set of words has exponential growth rate $\left(2k-3\right)$. 

The remaining words from $\overline{S_{k,N}}$, which we denote by
$Q_{k,N}$, can be described in a similar fashion to the cyclic words
from $M_{k,N}$. In a similar argument as in Lemma \ref{lem:growth-of-MkN},
we can shorten each word from $Q_{k,N}$ by the corresponding Whitehead
automorphisms until we get a word outside $Q_{k,N}$. Since we have
already seen that $\left|\overline{S_{k,N}}\setminus Q_{k,N}\right|$
has exponential growth rate $\left(2k-3\right)$, we can complete
the proof of Corollary \ref{cor:growth-of-proper-free-factors} in
the same manner we proved Lemma \ref{lem:growth-of-MkN}.$\qed$

\section{Most Triplets are Negligible\label{sec:Growth-of-partitions}}

The last section is dedicated to proving Proposition \ref{prop:Every-triplet-except...-is-negligible},
stating that for most triplets, the set $A_{Y,Z,v}$ has exponential
growth rate strictly smaller than $\left(2k-3\right)$. This is done
by way of considering different cases according to the cardinalities
$\left|Y\right|$ and $\left|Z\right|$, and treating each case separately.
To simplify the notation we denote $|Y|$ and $\left|Z\right|$ by
$y$ and $z$, respectively. Note that $y+z=2k-1$. The assumptions
of Proposition \ref{prop:Every-triplet-except...-is-negligible} are
that $y,z\geq2$ and $Y\neq\{x,x^{-1}\}$. The main technique is to
rely on the following intuitive lemma. We call a set of words in $\F_{k}$
\emph{Markovian} if it is closed under taking prefixes and if to every
$x\in X^{\pm1}$ corresponds a fixed subset $\Sigma_{x}\subseteq X^{\pm1}$
of letters which can follow $x$. Namely, if $w\in A$ is of length
$N$ and terminates with $x$, one can extend it to a word in $A$
of length $\left(N+1\right)$ by appending one of the letters from
$\Sigma_{x}$. Obviously, the sets $A_{Y,Z,v}$ are all Markovian
(to be precise, the set of all cyclically reduced representatives
of the cyclic words in some $A_{Y,Z,v}$ is Markovian).
\begin{lem}
\label{lem:technical}Let $A$ be a \emph{Markovian }set of words
in $\F_{k}$ and let $\alpha>1$. Assume that for each letter $x\in X^{\pm1}$
there is some $1\leq r=r\left(x\right)\in\mathbb{N}$ such that $x$
is followed by one of less than $\alpha^{r}$ possible $r$-tuples
of letters. Then the exponential growth of $A$ is less than $\alpha$.\end{lem}
\begin{proof}
For every $x\in X^{\pm1}$ and $1\leq i\leq r\left(x\right)$ let
$T_{x,i}$ denote the number of possible $i$-tuples which can follow
$x$ in words from $A$. (In particular, $T_{x,r\left(x\right)}<\alpha^{r\left(x\right)}$.)
For $w_{1},w_{2}\in A$ we say that $w_{2}$ is an \emph{$i$-extension
}of $w_{1}$ if $w_{1}$ is a prefix of $w_{2}$ and $\left|w_{2}\right|-\left|w_{1}\right|=i$.

Let $A_{N}$ be the set of words of length $N$ in $A$. Define a
subset $B\subseteq A$ by the following recursive rules: $A_{1}\subseteq B$
and if $w\in B$ terminates with $x$, then 
\begin{itemize}
\item all $i$-extensions of $w$ for $1\leq i\leq r\left(x\right)-1$ do
\emph{not }belong to $B$ (there are $T_{x,1}+T_{x,2}+\ldots+T_{x,r\left(x\right)-1}$
such words), and
\item all the $T_{x,r\left(x\right)}$ words which are $r\left(x\right)$-extensions
of $w$ belong to $B$.
\end{itemize}

Define $f:A\to\mathbb{R}$ as follows: for every $w\in B$ terminating
with the letter $x$,
\begin{itemize}
\item set f$\left(w\right)=1$, and
\item for every $1\leq i\leq r\left(x\right)-1$ and every $i$-extension
$u$ of $w$, set 
\[
f\left(u\right)=\frac{\left(T_{x,r\left(x\right)}\right)^{i/r\left(x\right)}}{T_{x,i}}<\frac{\alpha^{i}}{T_{x,i}}.
\]

\end{itemize}

Now, set $g\left(N\right)=\sum_{w\in A_{N}}f\left(w\right)$. For
$w\in B$ terminating with $x$ and $1\leq i\leq r\left(x\right)$,
the sum of $f$ over all $i$-extensions of $w$ is less than $\alpha$
times the sum over all $\left(i-1\right)$-extensions of $w$. We
obtain that $g\left(N+1\right)<\alpha\cdot g\left(N\right)$, so the
exponential growth rate of $g$ is $<\alpha$. We end the proof by
claiming that $c<\frac{g\left(N\right)}{\left|A_{N}\right|}<C$ for
some positive constants $c,C$. Indeed, one can set 
\[
c=\min_{\substack{x\in X^{\pm1}\\
1\leq i\leq r\left(x\right)
}
}\frac{\left(T_{x,r\left(x\right)}\right)^{i/r\left(x\right)}}{T_{x,i}},\,\,\,\,\, C=\max_{\substack{x\in X^{\pm1}\\
1\leq i\leq r\left(x\right)
}
}\frac{\left(T_{x,r\left(x\right)}\right)^{i/r\left(x\right)}}{T_{x,i}}.
\]

\end{proof}
We now return to the proof of Proposition \ref{prop:Every-triplet-except...-is-negligible}.
We shall use Lemma \ref{lem:technical} for the sets $A_{Y,Z,v}$
with $\alpha=2k-3$.

\noindent \textbf{Case 1:} $\boldsymbol{|Y|,|Z|\geq4}$: ~~~Let
$\left[w\right]\in A_{Y,Z,v}$ and let $x\in X^{\pm1}$ appear in
$w$. If $x\in Y$ then either the inverse of the following letter
is $v$ or it belongs to $Y\setminus\left\{ x\right\} $, so there
are $y$ options, and $y=2k-1-z\leq2k-5$. The same argument applies
for $x\in Z$ . Finally, the only letter that cannot follow $v$ is
$v^{-1}$, so $v$ is followed by:
\begin{itemize}
\item one of $y$ letters from $Y$, which are followed in turn by one of
$y$ letters, or 
\item one of $z-1$ letters from $Z\setminus\left\{ v^{-1}\right\} $, which
are followed in turn by one of $z$ letters, or 
\item $v$, which is followed by one of $2k-1$ letters.
\end{itemize}
Overall, there are $y^{2}+z\left(z-1\right)+\left(2k-1\right)$ possibilities
for the two letters following $v$. It is easy to see that under the
assumptions in the current case, this expression is largest when $y=2k-5$
and $z=4$. But even in this case, 
\[
y^{2}+z\left(z-1\right)+\left(2k-1\right)=\left(2k-5\right)^{2}+12+\left(2k-1\right)<\left(2k-3\right)^{2},
\]
 (note that if $y,z\geq4$ then $k\geq5$). This completes the proof.\\

\noindent \textbf{Case 2:} $\boldsymbol{|Y|=3}$:\label{case:y=00003D3}
~~~Assume first that $k\geq5$, so $2k-3\geq7$. As in the previous
proof, let $x\in X^{\pm1}$ be a letter in a word from $A_{Y,Z,v}$.
If $x\in Y$, the following letter is one of $y=3$ possibilities.
If $x\in Z$ there are at most $z=2k-4$ possibilities. Finally, if
$x=v$, then $v$ is followed by
\begin{itemize}
\item one of $3$ letters from $Y$, which are followed in turn by one of
$3$ letters, or 
\item one of $2k-5$ letters from $Z\setminus\left\{ v^{-1}\right\} $,
which are followed in turn by one of $2k-4$ letters, or 
\item $v$, which is followed by one of $2k-1$ letter.
\end{itemize}
Overall, there are $3^{2}+\left(2k-5\right)\left(2k-4\right)+\left(2k-1\right)$
possibilities for the two letters following $v$. This is strictly
less than $\left(2k-3\right)^{2}$ for $k\geq5$.

Suppose next that $k=4$, so now $2k-3=5$. Any $x\in X^{\pm1}\setminus\left\{ v\right\} $
is followed by at most $4$ possible letters. As for $v$ itself,
we need to distinguish between two cases: either $Y$ does not contain
a letter and its inverse, in which case w.l.o.g.~$Y=\left\{ a,b,c\right\} $
and $Z=\left\{ a^{-1},b^{-1},c^{-1},v^{-1}\right\} $; or w.l.o.g.~$Y=\left\{ a,a^{-1},b\right\} $
and $Z=\left\{ b^{-1},v^{-1},c^{\pm1}\right\} $. In the first case
an easy computation%
\footnote{Computation of this kind can be easily carried out in some Excel-type
spreadsheet program.%
} shows that $v$ is followed by one of $115$ possible triplets of
letters, and we are done as $115<5^{3}$. In the second case, $v$
is followed by one of 617 possible quadruplets, and $617<5^{4}$.

Finally, if $y=3$ and $k=3$ ($k$ cannot be smaller than $3$ if
$y=3$), then $2k-3=3$. The partition is, up to name changes, $Y=\left\{ a,a^{-1},b\right\} $
and $Z=\left\{ b^{-1},v^{-1}\right\} $. An easy computation shows
that any letter $x\ne v$ is followed by at most $8$ possible pairs
of letters, and $v$ is followed by at most $17{,}883$ possible $9$-tuples
(and $17{,}833<3^{9}=19{,}683$).%
\footnote{Alternatively, one can show that the exponential growth rate here
equals the largest real root of $\lambda^{5}-3\lambda^{4}+\lambda^{3}-\lambda^{2}-\lambda+7$,
which is about $2.68$.%
}\\

\noindent \textbf{Case 3: }$\boldsymbol{\left|Z\right|=3}$\label{case:z=00003D3}
~~~Assume first that $k\geq4$, hence $2k-3\geq5$. If $x\in Y$,
it is followed by the inverse of one of $2k-5$ letters from $Y\setminus\left\{ x\right\} $
or by $v^{-1}$, a total of $2k-4$ possibilities. A letter from $Z$
is followed by one of $3<2k-3$ letters ($v^{-1}$ and two letters
whose inverse belongs to $Z$). 

To analyze the number of possibilities after $v$, we distinguish
between two cases:
\begin{enumerate}
\item Assume that $Z$ contains a letter and its inverse, i.e. $Z=\left\{ v^{-1},a,a^{-1}\right\} $.
In this case,

\begin{itemize}
\item Every $x\in Y$, is followed either by one of $2k-5$ letters from
$Y\setminus\left\{ x^{-1}\right\} $ which are then followed by one
of $\left(2k-4\right)$ letters, or by $v^{-1}$ which is followed
by one of $3$ possible letters: a total of $\left(2k-5\right)\left(2k-4\right)+3=4k^{2}-18k+23$
possible pairs.
\item The letter $a$ is followed either by one of two letters from $Z$
($a$ or $v^{-1}$), which are in turn followed by one of $3$ possible
letters, or by $v$ which is followed by one of $2k-1$ letters: a
total of $2\cdot3+\left(2k-1\right)=2k+5$ possible pairs. The same
computation holds for $a^{-1}$.
\item Every pair of letters following $v$ starts either with some $x\in Z$
($2\cdot3$ options), by $v$ ($2k-1$ options) or by some $x\in Y$
($\left(2k-4\right)^{2}$ options): a total of $6+2k-1+\left(2k-4\right)^{2}=4k^{2}-14k+21$
possible pairs.
\end{itemize}

We can now count the number of possible triplets of letters following
$v$:
\begin{itemize}
\item $2\left(2k+5\right)$ triplets begin with $a$ or $a^{-1}$.
\item $\left(2k-4\right)\left(4k^{2}-18k+23\right)$ triplets begin with
some $x\in Y$.
\item $4k^{2}-14k+21$ triplets begin with $v$.
\end{itemize}

The total number of possible triplets following $v$ is, therefore,
$8k^{3}-48k^{2}+108k-61$ which is strictly less than $\left(2k-3\right)^{3}$
when $k\geq4$.

\item The other possibility is that $Z$ does not contain a letter and its
inverse, hence $Z=\left\{ v^{-1},a,b\right\} $. A similar computation
shows that $v$ is followed in this case by one of $8k^{3}-56k^{2}+160k-145$
which is again strictly less than $\left(2k-3\right)^{3}$ when $k\geq4$.
\end{enumerate}
Finally, if $k=3$ then $Z=\left\{ v^{-1},a,b\right\} $ and $Y=\left\{ a^{-1},b^{-1}\right\} $
(for otherwise $Y=\left\{ x,x^{-1}\right\} $). Another technical
computation shows that $a^{-1}$ and $b^{-1}$ are followed by at
most $2$ possible letters, $v^{-1}$ by at most $7$ possible pairs
of letters, $a$ and $b$ by at most $237<3^{5}$ $5$-tuples, and
$v$ by at most $41{,}372{,}449<3^{16}$ $16$-tuples of letters%
\footnote{Alternatively, the exponential growth rate here equals the largest
real root of $\lambda^{4}-2\lambda^{3}-4\lambda^{2}+2\lambda+7$,
which is about $2.85$.%
}.\\

\noindent \textbf{Case 4:} $\boldsymbol{\left|Y\right|=2}$\label{case:y=00003D2}
~~~Since $Y\neq\left\{ x,x^{-1}\right\} $, assume w.l.o.g.~that
$Y=\left\{ a,b\right\} $. The possibility $k=3$ was already dealt
with in the previous case, so assume $k\geq4$. Similar calculations
to those above show that:
\begin{itemize}
\item $a$ and $b$ are followed by one of $2<\left(2k-3\right)$ letters.
\item $a^{-1}$ and $b^{-1}$ are followed by one of $\left(4k^{2}-14k+16\right)<\left(2k-3\right)^{2}$
pairs of letters.
\item Any letter $x$ such that $x,x^{-1}\in Z$ is followed by one of $\left(4k^{2}-16k+21\right)<\left(2k-3\right)^{2}$
pairs of letters.
\item $v^{-1}$ is followed by one of $\left(4k^{2}-16k+19\right)<\left(2k-3\right)^{2}$
pairs of letters.
\item $v$ is followed by one of $\left(8k^{3}-44k^{2}+106k-91\right)$
triplets of letters. This is strictly less than $\left(2k-3\right)^{3}$
for $k\geq5$. For $k=4$, a concrete computation shows $v$ is followed
by one of $613<5^{4}$ possible $4$-tuples of letters.\\

\end{itemize}
\textbf{Case 5:} $\boldsymbol{\left|Z\right|=2}$\label{case:z=00003D2}~~~~Assume
w.l.o.g.~that $Z=\left\{ a,v^{-1}\right\} $. The case $k=3$ was
handled in the case $\left|Y\right|=3$, so assume $k\geq4$. Again,
the following formulas can be easily computed:
\begin{itemize}
\item $a$ and $v^{-1}$ are each followed by one of $2<\left(2k-3\right)$
letters.
\item $a^{-1}$ is followed by one of $\left(4k^{2}-14k+14\right)<\left(2k-3\right)^{2}$
pairs of letters.
\item Any letter $x$ such that $x,x^{-1}\in Y$ is followed by one of $\left(4k^{2}-16k+19\right)<\left(2k-3\right)^{2}$
pairs of letters.
\item $v$ is followed by one of $\left(8k^{3}-40k^{2}+80k-51\right)$ triplets
of letters. This is strictly less than $\left(2k-3\right)^{3}$ for
$k\geq6$. For $k=5$, a concrete computation shows $v$ is followed
by one of $1{,}951<7^{4}$ possible $4$-tuples of letters, and for
$k=4$, $v$ is followed by one of $557<5^{4}$ possible $4$-tuples
of letters.
\end{itemize}
This finishes the proof of Proposition \ref{prop:Every-triplet-except...-is-negligible}.
$\qed$

\section{Open Questions\label{sec:Open-Questions}}

Finally, we mention the following closely related questions which
are still open:

\begin{ques}What can be said about the growth of $\mathrm{Aut}\left(F_{k}\right)$
with respect to standard generating sets such as

i) Nielsen moves? 

ii) Whitehead automorphisms?\end{ques}

\begin{ques}What is the smallest possible exponential growth rate
of $P_{k,N}$ (or $C_{k,N}$) with respect to arbitrary finite generating
sets of $F_{k}$ (not necessarily bases)?\end{ques}

\begin{ques}What is the growth of $C_{k,N}\setminus L_{k,N}$? What
does a generic primitive cyclic element containing every letter at
least twice (or not at all) look like? Is it, up to permuting the
letters, of the form $x_{1}w\left(x_{1}x_{2},x_{3},x_{4},\ldots,x_{k}\right)$?
(In particular, the latter set of words shows that the growth of $C_{k,N}\setminus L_{k,N}$
is at least $\lambda_{k}$, the largest root of $\lambda^{3}-\left(2k-5\right)\lambda^{2}-\lambda-\left(2k-3\right)$,
which satisfies $\lambda_{k}\searrow2k-5$ as $k\to\infty$.)

\end{ques}

\begin{ques}What is the growth of other $\mathrm{Aut}\left(\F_{k}\right)$-orbits
in $\F_{k}$? Which orbits, other than that of the primitives, have
the largest growth? \\
We conjecture the following is true: For $w\in\F_{k}$, let $\mu\left(w\right)$
denote the minimal (positive) number of instances of some letter $x\in X$
in any element of the $\mathrm{Aut}\left(\F_{k}\right)$-orbit of
$w$ (this number does not depend on $x$). Then the growth of the
set of cyclic words in the orbit of $w$ is $\sqrt[\mu\left(w\right)]{2k-3}$.
If true, this shows that unless $w$ is primitive, most words in its
orbit are conjugates of small words, so that the growth of the orbit
is always $\sqrt{2k-1}$.\end{ques}

\bibliographystyle{amsalpha}
\bibliography{self,united}

\noindent Doron Puder, \\
Einstein Institute of Mathematics, \\
Edmond J. Safra Campus, Givat Ram\\
The Hebrew University of Jerusalem\\
Jerusalem, 91904, Israel\\
doronpuder@gmail.com\\

\noindent Conan Wu, \\
Department of Mathematics, \\
Princeton University\\
Fine Hall, Washington Road \\
Princeton NJ 08544-1000 USA

\noindent shuyunwu@princeton.edu
\end{document}